\documentclass[english]{amsart}
\usepackage{amssymb,amsmath}
\usepackage{amsthm, dsfont}
\usepackage{hyperref}
\usepackage[all]{xy}
\usepackage[active]{srcltx}

\def\GG{{\mathbb G}}
\def\NN{{\mathbb N}}
\def\ZZ{{\mathbb Z}}
\def\QQ{{\mathbb Q}}

\def\G{{\mathcal G}}
\def\H{{\mathcal H}}

\def\fH{{\mathfrak H}}
\def\fM{{\mathfrak M}}

\def\id{{\rm id}}

\def\cat#1{{\sf #1}}

\def\vect#1{\text{\boldmath $#1$\unboldmath}} 
\def\ie{i.\,e.}
\def\isom{\cong}

\def\ideal{\unlhd}

\DeclareMathOperator{\Ker}{Ker}

\DeclareMathOperator{\Aut}{Aut}
\DeclareMathOperator{\GL}{GL}
\DeclareMathOperator{\Gal}{Gal}
\DeclareMathOperator{\Quot}{Quot}
\DeclareMathOperator{\Spec}{Spec}
\def\uGal{\underline{\Gal}}

\def\markdef{\bf }

\theoremstyle{plain}
\newtheorem{thm}{Theorem}[section]
\newtheorem{cor}[thm]{Corollary}
\newtheorem{lem}[thm]{Lemma}

\theoremstyle{definition}
\newtheorem{defn}[thm]{Definition}
\newtheorem{exmp}[thm]{Example}
\newtheorem{rem}[thm]{Remark}

\newenvironment{notation}{{\bf Notation}\it}

\newtheoremstyle{Acknowledgements}% name
  {}% {\topsep}%      Space above
    {}% {\topsep}%      Space below
     {}%         Body font
     {}%         Indent amount (empty = no indent, \parindent = para indent)
    {\bfseries}% Thm head font
    {}%        Punctuation after thm head
     {.5em}%     Space after thm head: " " = normal interword space;
    {\thmname{#1}\thmnumber{ }\thmnote{ (#3)}}%  Thm head spec (can beleft
\theoremstyle{Acknowledgements}
\newtheorem{ack}{Acknowledgements.}

\begin{document}

\title[finite inverse problem]{On the finite inverse problem in iterative differential Galois
theory}

\author{Andreas Maurischat}
\address{Interdisciplinary
Center for  Scientific Computing~(IWR), Heidelberg 
University, Im Neuenheimer Feld 368, 69120 Heidelberg, Germany }
\email{andreas.maurischat@iwr.uni-heidelberg.de}
\urladdr{http://www.iwr.uni-heidelberg.de/~Andreas.Maurischat}
\thanks{}

% AMS-Classification

% Keywords

%-----------------------------------------------------------  
\begin{abstract}
In positive characteristic, nearly all Picard-Vessiot extensions are inseparable over some intermediate iterative differential extensions. In the Galois correspondence, these intermediate fields correspond to
nonreduced subgroup schemes of the Galois group scheme. Moreover, the
Galois group scheme itself may be nonreduced, or even infinitesimal.
In this article, we investigate which finite group schemes
occur as iterative differential Galois group schemes over a given ID-field.
For a large class of ID-fields, we give a description of all occuring finite group
schemes.
\end{abstract}
%-----------------------------------------------------------  

\maketitle
%---------------------------------------------------------------------
% Beginn des eigentlichen Artikels
%---------------------------------------------------------------------

\section{Introduction}

Picard-Vessiot theory for iterative differential extensions in positive
characteristic as conceived by Matzat and van der Put in \cite{bhm-mvdp:ideac}
was restricted to separable extensions and algebraically closed fields of
constants. This restriction was necessary, since the Galois group was given as (the rational
points of) a linear algebraic group. Furthermore, intermediate fields over which
the Picard-Vessiot ring is inseparable are not taken account in their Galois
correspondence.
In \cite{am:gticngg}, the Picard-Vessiot theory was extended to perfect fields
of constants and to inseparable extensions. This was made possible
by constructing the Galois group as an affine group scheme. In the case of a
separable PV-extension over an algebraically closed field of constants $C$, the
$C$-rational points of this group scheme are exactly the original Galois group as defined by Matzat and van der Put.

In this article, we solve the inverse problem for finite group schemes in
iterative Picard-Vessiot theory. That is, we give necessary and sufficient
conditions for a finite group scheme to be a Galois group scheme
over a given ID-field $F$. The general result which still depends on the inverse problem
in classical Galois theory over $F$ is given in
Theorem~\ref{thm:finite-realisation}. In the case where $F$ itself is a
PV-extension of an algebraic function field over algebraically closed
constants, this classical inverse problem is solved. Therefore, in that case, we can
explicitly describe the occuring Galois group schemes
(cf.~Corollary~\ref{cor:inverse-problem-over-pv-field}).

This article is structured as follows.
In Section \ref{sec:notation}, we describe the notation and results regarding
Picard-Vessiot theory in positive characteristic which are required later.
Section \ref{sec:insep-pv} is dedicated to the case of infinitesimal Galois 
group schemes (i.e.~purely inseparable PV-extensions), whereas results for finite
reduced Galois group schemes are given in Section \ref{sec:sep-pv}. Results from both sections are
used in Section \ref{sec:finite-pv} to solve the inverse problem for
finite group schemes.

\begin{ack}
I thank B.~H.~Matzat for drawing my attention to the inverse problem for
nonreduced group schemes. I also thank E.~Dufresne for helpful comments to
improve this paper.
\end{ack}

\section{Basic notation}\label{sec:notation}

All rings are assumed to be commutative with unit.
We will use the following notation (see also \cite{am:igsidgg}). An iterative 
derivation on a ring $R$ is a homomorphism of rings $\theta:R\to
R[[T]]$, such that $\theta^{(0)}=\id_R$ and for all $i,j\geq
0$, $\theta^{(i)}\circ \theta^{(j)}=\binom{i+j}{i}\theta^{(i+j)}$,
where the maps $\theta^{(i)}:R\to R$ are defined by
$\theta(r)=:\sum_{i=0}^\infty \theta^{(i)}(r)T^i$.
The pair $(R,\theta)$ is then called an ID-ring and
$C_R:=\{ r \in R\mid \theta(r)=r\}$ is called the {\markdef ring of
  constants } of $(R,\theta)$. An iterative
derivation $\theta$ resp. the ID-ring $(R,\theta)$ is called {\markdef nontrivial}, if $C_R\ne R$.
An ideal $I\ideal R$ is called an
 {\markdef ID-ideal} if $\theta(I)\subseteq I[[T]]$ and $R$ is
 {\markdef ID-simple} if $R$ has no proper nontrivial
 ID-ideals. An ID-ring which is a field is called an {\markdef ID-field}. 
Iterative derivations are extended to localisations by
 $\theta(\frac{r}{s}):=\theta(r)\theta(s)^{-1}$ and to tensor products
 by 
$$\theta^{(k)}(r\otimes s)=\sum_{i+j=k} \theta^{(i)}(r)\otimes
\theta^{(j)}(s)$$ 
for all $k\geq 0$.

If $R$ is an integral domain of positive characteristic $p$, the iterative
derivation induces a family of ID-subrings given by
$$R_\ell:=\bigcap_{0<j<p^\ell} \Ker(\theta^{(j)}),$$
($\ell\in\NN$) and also a family of ID-overrings $R_{[\ell]}:=(R_\ell)^{p^{-\ell}}$ ($\ell\in\NN$) in some
inseparable closure with iterative derivation given by
$$\theta_{R_{[\ell]}}(x):=\left(\theta_R(x^{p^\ell})\right)^{p^{-\ell}}.$$

\begin{notation}
From now on, $(F,\theta)$ denotes an ID-field of positive characteristic $p$,
and $C=C_F$ its field of constants. We assume that $C$ is a
perfect field, and that $\theta$ is {\em non-degenerate}, i.\,e., that
$\theta^{(1)}\ne 0$. 
\end{notation}

\begin{rem}\label{rem:max-id-extension}
In this setting, $F_\ell$ is an ID-subfield with $C_{F_\ell}=C$,
and since $\theta^{(p^{\ell})}$ is a nilpotent $F_{\ell+1}$-linear
endomorphism of $F_{\ell}$ and $\theta^{(p^{\ell})}$ is of nilpotence order
$p$ with
$1$-dimensional kernel, one has $[F_\ell:F_{\ell+1}]=p$. Therefore, one obtains
$[F:F_\ell]=p^\ell$ for all $\ell$.

Furthermore, $F_{[\ell]}$ is an ID-extension of $F$ with same constants (since
$C$ is perfect), and $[F_{[\ell]}:F]=p^{\ell(m-1)}$, where $m$ denotes the
degree of imperfection of $F$ (possibly infinite).

In \cite{am:igsidgg}, Prop. 4.1, it is shown that $F_{[\ell]}/F$ is the maximal
purely inseparable ID-extension of $F$ of exponent $\leq \ell$.
\end{rem}

We now recall some definitions from Picard-Vessiot theory:
\begin{defn}
Let $A=\sum_{k=0}^\infty A_k T^k\in \GL_n(F[[T]])$ be a matrix with
$A_0=\mathds{1}_n$ and for all $k,l\in\NN$, 
$\binom{k+l}{l}A_{k+l}=\sum_{i+j=l} \theta^{(i)}(A_{k})\cdot A_j.$
An equation
$$\theta(\vect{y})=A\vect{y},$$
where $\vect{y}$ is a vector of indeterminants, is called an {\markdef
  iterative differential equation (IDE)}.
\end{defn}

\begin{rem}\label{ide-condition}
The condition on the $A_{k}$ is equivalent to the condition that
$\theta^{(k)}(\theta^{(l)}(Y_{ij}))=
\binom{k+l}{k}\theta^{(k+l)}(Y_{ij})$ holds
for a matrix $Y=(Y_{ij})_{1\leq i,j\leq n}\in\GL_n(E)$ satisfying $\theta(Y)=AY$, where $E$ is some ID-extension of
$F$. (Such a $Y$ is called a {\markdef fundamental solution matrix}) . The condition
$A_{0}=\mathds{1}_n$ is equivalent to
$\theta^{(0)}(Y_{ij})=Y_{ij}$, and already implies that the matrix $A$ is
invertible.
\end{rem}

\begin{defn}
An ID-ring $(R,\theta_R)\geq (F,\theta)$ is called a {\markdef
  Picard-Vessiot ring} (PV-ring) for the IDE $\theta(\vect{y})=A\vect{y}$, if
the following holds: 
\begin{enumerate}
\item $R$ is an ID-simple ring.
\item There is a fundamental solution matrix $Y\in\GL_n(R)$, \ie{}, an invertible
  matrix satisfying $\theta(Y)=AY$.
\item As an $F$-algebra, $R$ is generated by the coefficients of $Y$
  and by $\det(Y)^{-1}$.
\item $C_R=C_F=C$.
\end{enumerate}
The quotient field $E=\Quot(R)$ (which exists, since such a PV-ring
is always an integral domain) is called a {\markdef
  Picard-Vessiot field} (PV-field) for the IDE
$\theta(\vect{y})=A\vect{y}$.\footnote{The PV-rings and PV-fields defined
here were called pseudo Picard-Vessiot rings (resp. pseudo Picard-Vessiot fields)
in \cite{am:gticngg} and \cite{am:igsidgg}. This definition, however, is the
most natural generalisation of the original definition of PV-rings and
PV-fields to non algebraically closed fields of constants.}
\end{defn}

For a PV-ring $R/F$ one defines the functor
$$\underline{\Aut}^{ID}(R/F): (\cat{Algebras} / C) \to (\cat{Groups}),
D\mapsto \Aut^{ID}(R\otimes_C D/F\otimes_C D)$$
where $D$ is equipped with the trivial iterative derivation.
In \cite{am:gticngg}, Sect.~10, it is shown that this functor is
represent\-able by a $C$-algebra of
finite type, and hence, is an affine group scheme of finite type over 
$C$. This group scheme is called the (iterative differential) {\markdef Galois
  group scheme} of the extension $R$ over $F$ -- denoted by
$\underline{\Gal}(R/F)$ --, or also, the
Galois group scheme of  the extension $E$ over $F$, $\underline{\Gal}(E/F)$,
where
$E=\Quot(R)$ is the corresponding PV-field.

Furthermore, $\Spec(R)$ is a $(\uGal(R/F) \times_{C}F)$-torsor and the
corresponding isomorphism of rings
\begin{equation}\label{eq:torsor-iso} \gamma:R\otimes_F R\to R\otimes_C
C[\uGal(R/F)]\end{equation}
is an $R$-linear ID-isomorphism. Again, $C[\uGal(R/F)]$ is equipped with
the trivial iterative derivation.

This torsor isomorphism (\ref{eq:torsor-iso}) is the key tool to establish the Galois
correspondence between
the closed subgroup schemes of $\G=\uGal(R/F)$ and the intermediate ID-fields
of the extension $E/F$, in more detail:
\begin{thm}{\bf (Galois correspondence)}\label{galois_correspondence}
Let $E/F$ be a PV-extension with PV-ring $R$ and Galois group scheme
$\G$.

Then there is an inclusion reversing bijection between
$$\fH:=\{ \H \mid \H\leq\G \text{ closed subgroup scheme of }\G
\}$$
and
$$\fM:=\{ M \mid F\leq M\leq E \text{ intermediate ID-field} \}$$
given by 
$\Psi:\fH \to \fM,\H\mapsto E^{\H}$ and 
$\Phi:\fM \to \fH, M\mapsto \underline{\Gal}(E/M)$.\\
With respect to this bijection, $\H\in \fH$ is a normal subgroup of $\G$, if and
only if $E^{\H}$ is a PV-field over $F$. In this case the Galois
group scheme $\underline{\Gal}(E^{\H}/F)$ is isomorphic to $\G/\H$.
\end{thm}

(See \cite{am:gticngg}, Thm.~11.5, resp. \cite{am:igsidgg}, Prop.~3.4 and 
Thm.~3.5 for the proof of this theorem.) 
The invariants $E^\H$ are defined to be all elements $e=\frac{r}{s}\in E$, such
that, for all 
$C$-algebras $D$, and all $h\in\H(D)$,
$$\frac{h(r\otimes 1)}{h(s\otimes 1)}=e\otimes 1\in\Quot(E\otimes_C D),$$
where $\Quot(E\otimes_C D)$ denotes the localisation of $E\otimes_C D$ by all
nonzerodivisors.

The following table shows some properties of the Galois group scheme and the
corresponding properties of the PV-extension $E/F$:

\medskip

\centerline{\begin{tabular}{c|c}
{\bf Property of $\G$} & {\bf Property of $E/F$}\\[1mm]
\hline \\[-2mm]
finite scheme with $\dim_C(C[\G])=m$ & finite extension with $[E:F]=m$.\\[1mm]
reduced scheme & separable extension\\[1mm]
infinitesimal scheme of height $n$ & purely inseparable extension
of exponent $n$.
\end{tabular}}

\vspace*{5mm}

The first correspondence is obtained by comparing dimensions in the
torsor-isomorphism, since for finite extensions the PV-ring already is a field.
The other two are proved in \cite{am:gticngg} and \cite{am:igsidgg}, respectively.

\section{Purely inseparable PV-extensions}\label{sec:insep-pv}

In this section, we are interested in PV-extensions with infinitesimal Galois
group schemes, i.e., finite purely inseparable PV-extensions. Since the
results presented here  already appear in \cite{am:igsidgg},  most
proofs are only sketched.

\begin{thm}\label{frob-pullback} \cite[Thm.~4.2]{am:igsidgg}
Let $E$ be a PV-extension of $F$, and let $\ell\in\NN$. Then
$E_{[\ell]}/F_{[\ell]}$ is a PV-extension, and its Galois group
scheme is related to $\uGal(E/F)$ by $({\bf
  Frob}^\ell)^{*}\left(\uGal(E_{[\ell]}/F_{[\ell]})\right)\isom
\uGal(E/F)$, where ${\bf Frob}$ denotes the Frobenius
morphism on $\Spec(C)$.
\end{thm}

\begin{proof}
We only sketch the proof, see  \cite{am:igsidgg} for details.

Let $R\subseteq E$ be the corresponding PV-ring. Since the iterative
derivation is non-degenerate on $F$, on has
$[F:F_\ell]=p^\ell=[E:E_\ell]$.
This implies that there exists a fundamental solution matrix $Y\in
\GL_n(E_\ell)$ for some IDE $\theta(\vect{y})=A\vect{y}$ defining the
PV-extension.
One then shows that $A\in \GL_n(F_\ell[[T^{p^\ell}]])$, and hence,
$E_\ell/F_\ell$ is a PV-extension with Galois group scheme
$\uGal(R_\ell/F_\ell)\isom \uGal(R/F)$.
By taking $p^{\ell}$-th roots, one obtains that $R_{[\ell]}$ is a
PV-ring over $F_{[\ell]}$ with fundamental solution matrix
$\left((Y_{i,j})^{p^{-\ell}}\right)_{i,j}$.
This gives the desired property:
\begin{equation*}({\bf
Frob}^\ell)^{*}\left(\uGal(E_{[\ell]}/F_{[\ell]})\right)\isom
\uGal(E_\ell/F_\ell)
\isom\uGal(E/F). \qedhere
\end{equation*}
\end{proof}

From this theorem we obtain a criterion for $F_{[\ell]}/F$ being a
PV-extension.

\begin{cor}\label{cor:F_ell-is-pv} \cite[Cor.~4.3]{am:igsidgg}
Let $F$ be an ID-field which is a PV-extension of some nontrivial ID-field $L$ satisfying $L_1=L^p$.
Then $F_{[\ell]}$ is a PV-extension of $F$, for all $\ell\in\NN$.
\end{cor}

\begin{proof}
We first remark, that the iterative derivation on $L$ is non-degenerate, since otherwise
$L^p=L_1=L$ and hence $L=L^{p^\ell}\subseteq \Ker(\theta^{(p^\ell)})$ for all $\ell$. This would
imply that $\theta$ is trivial.

By Remark~\ref{rem:max-id-extension}, the condition $L_1=L^p$ implies that
$L_{[\ell]}=L$ for all $\ell$. Hence,
by the previous theorem, $F_{[\ell]}/L$ is a PV-extension, and
therefore, $F_{[\ell]}/F$ is a PV-extension.
\end{proof}

\begin{thm}\label{thm:general-realisation}
Let $F$ be an ID-field with $C_F=C$ perfect, and suppose that $F$ has finite degree of
imperfection.\\
Let $\tilde{C}_{\ell}$ denote  the
maximal subalgebra of $C_{F_{[\ell]}\otimes_F F_{[\ell]}}$ which is a
Hopf algebra with respect to the comultiplication induced by 
$$F_{[\ell]}\otimes_F F_{[\ell]}\longrightarrow
\left(F_{[\ell]}\otimes_F F_{[\ell]}\right)\otimes_{F_{[\ell]}}\left(
F_{[\ell]}\otimes_F F_{[\ell]}\right), a\otimes b\mapsto (a\otimes
1)\otimes (1\otimes b).$$

Then an infinitesimal group scheme of height $\leq \ell$ is realisable as
a Galois group scheme over $F$, if and only if it is a factor group
of $\Spec(\tilde{C}_{\ell})$.

In particular, if $F_{[\ell]}/F$ is a PV-extension, then $\tilde{C}_{\ell}=
C_{F_{[\ell]}\otimes_F F_{[\ell]}}$, and
$\Spec(\tilde{C}_{\ell})$ is isomorphic to $\uGal(F_{[\ell]}/F)$.
\end{thm}

\begin{proof}
The first statement is proved in \cite{am:igsidgg}, Thm.~4.5. In the case where
$F_{[\ell]}/F$ is a PV-extension, the torsor isomorphism
(\ref{eq:torsor-iso}) implies that
$C[\uGal(F_{[\ell]}/F)]\isom C_{F_{[\ell]}\otimes_F F_{[\ell]}}$ is a Hopf
algebra. Hence, $\tilde{C}_{\ell}=
C_{F_{[\ell]}\otimes_F F_{[\ell]}}$, and $\Spec(\tilde{C}_{\ell})\isom
\uGal(F_{[\ell]}/F)$.
\end{proof}

\begin{cor}\label{cor:special-realisation} \cite[Cor.~4.6]{am:igsidgg}
Let $F$ be an ID-field and suppose that $F$ is a PV-extension of
some nontrivial ID-field $L$ satisfying $L_1=L^p$. 
An infinitesimal group scheme of height $\leq \ell$ is realisable as
ID-Galois group scheme over $F$, if and only if it is a factor group
of $\uGal(F_{[\ell]}/F)\isom \Spec(C_{F_{[\ell]}\otimes_F F_{[\ell]}})$.
\end{cor}

\begin{proof}
By Corollary \ref{cor:F_ell-is-pv}, $F_{[\ell]}/F$ is a PV-extension, and the
claim follows from Theorem \ref{thm:general-realisation}.
\end{proof}

\section{Finite separable PV-extensions}\label{sec:sep-pv}

In this section we consider PV-extensions with finite reduced Galois group
schemes, i.e., finite separable PV-extensions. Since iterative derivations
extend uniquely to finite separable field extensions (see
\cite{bhm-mvdp:ideac},2.1,(5)), we obtain a close
relationship to classical Galois extensions.

\begin{lem}
Let $E$ be a finite (classical) Galois extension of $F$ which is geometric over
$C$, and let $G$ be its Galois group. Let $E$ be equipped with the
unique iterative derivation extending the iterative derivation on $F$. Then the
extension $E/F$ is a PV-extension with Galois group scheme
$\uGal(E/F)=\Spec(C[G])$, the constant group scheme corresponding to $G$.
\end{lem}

\begin{proof}
Let $\{x_1,\dots, x_n\}$ be an $F$-basis of $E$, and let $a_{ij}\in F[[T]]$
such that $\theta(x_i)=\sum_{j=1}^n a_{ij}x_j$ for all $i=1,\dots, n$. Furthermore, let
$G=\{\sigma_1,\dots, \sigma_n\}$, and $Y=\bigl(\sigma_k(x_i)\bigr)_{1\leq i,k\leq n}\in
\GL_n(E)$. ($Y$ is invertible by Dedekind's lemma on the independence of automorphisms.)\\
By definition, one has $\theta(\left(\begin{smallmatrix}  x_1\\ \vdots\\ x_n 
\end{smallmatrix}\right))=A\left(\begin{smallmatrix}  x_1\\ \vdots\\ x_n 
\end{smallmatrix}\right)$, where $A=(a_{ij})
%_{1\leq i,j\leq n}
\in {\rm Mat}_{n\times n}(F[[T]])$.
Since the extension of the iterative derivation of $F$ to $E$ is unique, all automorphisms are
indeed ID-automorphisms, and therefore $\theta(Y)=AY$.

Therefore by Remark \ref{ide-condition}, $\theta(\vect{y})=A\vect{y}$ is an IDE and
$Y$ is a fundamental solution matrix for this IDE. Furthermore, $E$ is generated by the
entries of $Y$, and $C_E=C$, since $E/F$ is geometric.
Hence, $E$ is a PV-field for the IDE $\theta(\vect{y})=A\vect{y}$.\\
Finally, $\Spec(C[G])$ is a subgroup scheme of $\uGal(E/F)$, since $G$ acts by ID-auto\-morphisms,
and $\dim_C(C[G])=n=[E:F]=\dim_C(C[\uGal(E/F)])$. Hence, $\Spec(C[G])=\uGal(E/F)$.
\end{proof}

\begin{rem}
In the case that $C$ is algebraically closed, the statement also follows easily from
\cite{bhm-mvdp:ideac}, 4.1.:
in this case, Matzat and van der Put proved that the $C$-rational points of $\uGal(E/F)$
equal $G$. Since for an algebraically closed field $C$, a reduced group scheme
is determined by its $C$-rational points, the claim follows.
\end{rem}

\begin{lem}
Let $E/F$ be a finite separable PV-extension with Galois group scheme $\G$. Let
$E_{\bar{C}}=E\otimes_C \bar{C}$ and $F_{\bar{C}}=F\otimes_C \bar{C}$ be the
extensions of constants, where $\bar{C}$ denotes an algebraic closure of $C$.
Then $E_{\bar{C}}/F_{\bar{C}}$ is a finite (classical) Galois extension with
Galois group $\G(\bar{C})$.
\end{lem}

\begin{proof}
By definition of the Galois group scheme $\G$, the group
$\G(\bar{C})$ is equal to $\Aut^{ID}(E_{\bar{C}}/F_{\bar{C}})$. (Recall that in the
finite case the PV-ring and the PV-field are equal.) Since
$E_{\bar{C}}/F_{\bar{C}}$ is separable, $\G$ is reduced and hence,
$E^{\G(\bar{C})}=E^\G=F$. Therefore, $(E_{\bar{C}})^{\G(\bar{C})}=F_{\bar{C}}$, which
implies that $E_{\bar{C}}$ is Galois over $F_{\bar{C}}$ with Galois
group $\G(\bar{C})$.
\end{proof}

\begin{rem}
The previous lemma tells us that all finite separable PV-extensions become
classical Galois extensions after an algebraic extension of the constants.
Actually, this extension of constants can be chosen to be finite Galois. 
Hence finite separable PV-extensions are {\em almost
classical Galois extensions} in the sense of Greither and Pareigis 
(cf.~\cite{cg-bp:hgtsfe}, Def.~4.2).
\end{rem}

\section{Finite PV-extensions}\label{sec:finite-pv}

We now consider the case of arbitrary finite PV-extensions, i.e.,
PV-exten\-sions with finite Galois group schemes. By \cite{md-pg:ga}, Ch.~II,
\S 5, Cor.~2.4, every finite group scheme is the semi-direct product of an
infinitesimal group scheme and a finite reduced group scheme. Hence, the results
of the previous sections also give us information in this case.

\begin{thm}\label{thm:finite-realisation}
Let $\G$ be a finite group scheme over $C$, $\G^0\ideal \G$ the connected component of
$\G$ (an infinitesimal group scheme), and $\H\leq \G$ the induced reduced group
scheme. Assume that there is $\ell\geq {\rm ht}(\G^0)$ such that $F_{[\ell]}/F$
is a PV-extension. Then $\G$ is realisable over $F$, if and only if $\G\isom
\G^0\times \H$, $\G^0$ is a factor group of $\uGal(F_{[\ell]}/F)$ and $\H$ is
realisable over $F$.
\end{thm}

\begin{proof}
Let $\G\isom \G^0\times \H$, such that $\G^0$ is a factor group of
$\uGal(F_{[\ell]}/F)$ and $\H$ is realisable over $F$ as $\H\isom
\uGal(E''/F)$. By Theorem \ref{thm:general-realisation}, $\G^0$ is the Galois
group scheme of some intermediate PV-field $F\leq E'\leq F_{[\ell]}$. Since
$E''/F$ is separable and $E'/F$ is purely inseparable, $E'$ and $E''$ are linearly
disjoint over $F$, and so $E'\otimes_F E''$ is a PV-extension of $F$ with
Galois group scheme $\uGal(E'\otimes_F E''/F)\isom \G^0\times \H$. Hence, $\G$
is realisable over $F$.

On the other hand, let $\G$ be realised over $F$ as $\G\isom \uGal(E/F)$. By
\cite{md-pg:ga}, Ch.~II, \S 5, Cor.~2.4, $\G$ is a semi-direct product
$\G\isom \G^0\rtimes \H$, and therefore $\H\isom \G/\G^0\isom
\uGal(E^{\G^0}/F)$, i.e., $\H$ is realisable over $F$. Furthermore, $E^\H$ is a
purely inseparable ID-extension of $F$ of height $\leq {\rm ht}(\G^0)$. By
assumption, there is $\ell\geq {\rm ht}(\G^0)$ such that  $F_{[\ell]}/F$ is a
PV-extension and therefore $F_{[\ell]}$ is a PV-extension containing $E^\H$. As
in the first part of the proof, $\tilde{E}:=F_{[\ell]}\otimes_F E^{\G^0}$ is a
PV-extension of $F$ with Galois group $\uGal(\tilde{E}/F)\isom
\uGal(F_{[\ell]}/F) \times \uGal(E^{\G^0}/F)\isom \uGal(F_{[\ell]}/F) \times
\H$. Since $E^\H$ and $E^{\G^0}$ are subfields of $\tilde{E}$, $E$ is also a
subfield of $\tilde{E}$. Therefore, $\uGal(E/F)\isom \G^0\rtimes \H$ is a
factor group of $\uGal(\tilde{E}/F)\isom \uGal(F_{[\ell]}/F) \times
\H$ which implies that $\H$ acts also trivially on $\G^0$, i.e., the
semi-direct product $\G^0\rtimes \H$ is in fact a direct product. Finally,
we obtain that $E^\H$ is a PV-extension of $F$ (since $\H\leq \G$ is a
normal subgroup) with Galois group $\G^0$, and hence $\G^0$ is a factor
group of $\uGal(F_{[\ell]}/F)$.
\end{proof}

\begin{cor}\label{cor:inverse-problem-over-pv-field}
Let $C$ be algebraically closed, and let $F$ be a PV-extension of some function
field $L/C$ in one variable with non-degenerate iterative derivation. Then the finite group schemes
which occur as Galois
group scheme over $F$ are exactly the direct products $\G^0\times \H$, where
$\H$ is a constant group scheme (i.e., a reduced finite group scheme) and $\G^0$ is
a factor group of some $\uGal(F_{[\ell]}/F)$.
\end{cor}

\begin{proof}
By Corollary~\ref{cor:F_ell-is-pv}, $F_{[\ell]}$ is
a PV-extension of $F$ for all $\ell$. So
by Theorem~\ref{thm:finite-realisation}, we only have to show that every
finite reduced group scheme $\H$ is realisable. Since $C$ is algebraically
closed, the PV-extensions $E$ of $F$ with Galois group $\H$ are the (classical)
Galois extensions with Galois group $\H(C)$. By \cite{dh:fgepcp}, Thm.~4.4, the absolute
Galois group of $L$, $\Gal(L^{\rm sep}/L)$, is a free group on infinitely many generators.
Hence, there is an epimorphism $\phi:\Gal(L^{\rm sep}/L)\to \H(C)\times \Gal(F\cap L^{\rm sep}/L)$
such that the composition of $\phi$ and the projection ${\rm pr}_2$ onto the second factor is the
restriction map $\Gal(L^{\rm sep}/L) \to \Gal(F\cap L^{\rm sep}/L)$.

But this means that ${\rm pr}_1\circ \phi:\Gal(L^{\rm sep}/L)\to \H(C)$ corresponds to a Galois
extension $\tilde{L}$ of $L$ with group $\H(C)$ which is linearly disjoint to $F$. Hence
$\tilde{L}\otimes_L F$ is a Galois extension of $F$ with Galois group $\H(C)$.
\end{proof}

\begin{exmp}
Let $C$ be an algebraically closed field of positive characteristic $p$. We want $L/C$ to be a function field in one variable over $C$ with a non-degenerate
iterative derivation $\theta$, and $F$ to be a PV-extension of $L$ with Galois group scheme
$\GG_m$.
For example, we may take $L=C(t)$ with $\theta=\theta_t$
the iterative derivation with respect to $t$, given by
$\theta_t(t)=t+1\cdot T\in L[[T]]$,
and $F=L(t^\alpha)$, with $\alpha\in \ZZ_p\setminus \QQ$ and the
iterative derivation given by
$\theta^{(n)}(t^\alpha)=\binom{\alpha}{n}t^\alpha/t^n$. 

By Theorem \ref{thm:general-realisation}, for all $\ell\geq 0$, $F_{[\ell]}/L$
is a PV-extension with $\uGal(F_{[\ell]}/L)\isom \GG_m$, and the
``restriction map'' $\uGal(F_{[\ell]}/L)\isom \GG_m\to \uGal(F/L)\isom
\GG_m$ is given by the Frobenius map $x\mapsto x^{p^\ell}$.
Hence, $\uGal(F_{[\ell]}/F)\isom \mu_{p^\ell}$, the ``group of $p^\ell$th roots
of unity''. The only factor groups of $\mu_{p^\ell}$ are $\mu_{p^k}$ where
$k\leq \ell$. Hence by Theorem \ref{thm:finite-realisation}, the finite Galois group
schemes over $F$ are exactly the group schemes of the form
$\mu_{p^\ell}\times H$, where $\ell\geq 0$ and $H$ is finite reduced.
\end{exmp}

%---------------------------------------------------------------------
% Ende des eigentlichen Artikels
%---------------------------------------------------------------------
\bibliographystyle{plain}

\vspace*{.5cm}

\end{document}